\documentclass[11pt]{amsart}
\usepackage{amsmath,amssymb,amsthm,amscd}\usepackage[usenames]{color}
\usepackage[all]{xy}
\usepackage{enumerate}
\xyoption{all}

\newtheorem{thm}{Theorem}[section]
\newtheorem{lem}[thm]{Lemma}

\newtheorem{prop}[thm]{Proposition}

\theoremstyle{definition}
\newtheorem{rem}[thm]{Remark}

\newcommand{\cO}{{\mathcal O}}

\newcommand{\cA}{{\mathcal A}}

\newcommand{\cP}{{\mathcal P}}
\newcommand{\cR}{{\mathcal R}}

\newcommand{\lra}{\longrightarrow}
\newcommand{\ra}{\rightarrow}

\newcommand{\PP}{{\mathbb P}}

\DeclareMathOperator{\Sym}{Sym}

\DeclareMathOperator{\Ker}{Ker}
\DeclareMathOperator{\Nm}{Nm}
\DeclareMathOperator{\Pic}{Pic}

\title[Global Prym-Torelli for $r\ge 6$]{Global Prym-Torelli  for double coverings ramified in at least $6$ points}

\author[J.C. Naranjo and A. Ortega]{Juan Carlos Naranjo and Angela Ortega}

\address{J.C. Naranjo \\ Departament de Matem\`atiques i Inform\`atica \\ Universitat de Barcelona \\Spain }
\email{jcnaranjo@ub.edu}
              
\address{A. Ortega \\ Institut f\"ur Mathematik, Humboldt Universit\"at zu Berlin \\ Germany}
\email{ortega@math.hu-berlin.de}

\thanks{The first author was partially supported by the Proyecto de Investigaci\'on MTM2015-65361-P}
\date{}
\begin{document}

\begin{abstract}
 We prove that the ramified Prym map $\mathcal P_{g, r}$ which sends a covering $\pi:D\lra C$ ramified in $r$ points to the Prym variety $P(\pi):=\Ker(\Nm_{\pi })$ 
 is an embedding for all $r\ge 6$ and for all $g(C)>0$. Moreover, by studying the restriction to the locus of coverings of hyperelliptic curves, we show that 
 $\mathcal P_{g, 2}$ and $\mathcal P_{g, 4}$ have positive dimensional fibers.
 \end{abstract}

 \maketitle
 
\section{Introduction}

The ramified Prym map $\mathcal P_{g,r}$ assigns to a degree $2$ morphism  $\pi: D \lra C$ of smooth complex irreducible curves  ramified in $r>0$ 
points, a polarized abelian variety $P(\pi)$ of dimension $g - 1 + \frac r2 $, where $g$ is the genus of $C$. We assume that $g>0$ throughout  the paper.  

The variety $P(\pi)$ is called the Prym variety of $\pi$ and is defined as the kernel of the norm map $\Nm_{\pi }:JD \lra JC.$ 
Hence, by denoting by $\cR_{g,r} $ the moduli space of  isomorphism classes of the morphisms $\pi$, we have maps:
\begin{eqnarray*}
\cP_{g,r} : \cR_{g,r} & \lra & \cA^\delta_{g-1+\frac{r}{2}},
\end{eqnarray*}
to the moduli space of abelian varieties of dimension  $g-1+\frac{r}{2}$ with polarization type
 $\delta:=(1,\ldots, 1,2, \ldots ,2)$, with $2$  repeated $g$ times.

This is the analogous version of the classical (or ``unramified'') Prym maps ($r=0$) and has  attracted considerable attention in the last years. 
Combining the results 
contained in \cite{mp}, \cite{mn}, \cite{no} and \cite{nr} we know that the ramified Prym map is generically injective as far as the dimension of $\cR_{g,r}$ is 
less or equal to the dimension of the dimension of  $\cA^\delta_{g-1+\frac{r}{2}}$ with only one exception: when $r=4, g=3$ the map has degree $3$ 
(see \cite{nr}, \cite{bcv}).

One could expect that, like in the unramified case, generalized tetragonal constructions (as defined by Donagi in \cite{do}) or other procedures  
(as, for example, in \cite{il}, \cite{lo}) would 
give examples of non injectivity for these maps. In fact, the computation of the degree of $\cP_{3,4}$  in \cite{nr} is based on a tetragonal construction 
for coverings ramified in $4$ points,  such that the branch divisor belongs to the tetragonal linear series. For this reason, the recent work of Ikeda (see \cite{ik}) is
 somehow unexpected: he proves that 
$\mathcal P_{1,r}$ is injective for $r\ge 6$. Hence it is natural to explore if, contrary to the initial intuition inspired by the classical case, a global 
Torelli theorem holds for $\mathcal P_{g, r}$ for some other values of $g$ and $r$. The aim of this article is to answer completely this question by proving 
the following results:

\begin{thm}\label{main_embedding}
 The Prym map $\mathcal P_{g,r}$ is an embedding for all $r\ge 6$ and  all $g>0$. 
\end{thm}

\vskip 3mm
\begin{thm}\label{main_hyp} Let $\mathcal P^h_{g,r}$ be the restriction of $\mathcal P_{g,r}$ to the locus of coverings of hyperelliptic curves of genus $g$ ramified in $r$ points. Then 
\begin{enumerate}
\item [a)] The generic fiber of the map $\mathcal P^h_{g,2}$  is birational to a projective plane. 
\item [b)] The generic fiber of the map $\mathcal P^h_{g,4}$ is birational to an elliptic curve. 
\end{enumerate}
Thus all the ramified Prym maps with either $r=2$ or $r=4$ have positive dimensional fibers.
\end{thm}
Observe that the maps $\mathcal P_{1,2}$ and $\mathcal P_{1,4}$ have positive dimensional fibers by dimensional reasons, hence we will assume during the proof that $g\ge 2$. 

The paper is organized as follows:
In section $2$ we give a review of some results and we give a detailed analysis of the injectivity of the differential of the Prym map $\mathcal P_{g,r}$. 
We obtain that $dP_{g,r}(\pi )$ is injective in all the elements $\pi: D\lra C$ for $r\ge 6$.
The proof of the first theorem is completed in section $3$. The  main idea is that  the covering $\pi$ can be recovered  from the base locus of 
the linear system $|\Xi | $ of the theta divisor of the Prym variety $P(\pi)$. 
 In section $4$  we investigate the restriction of the Prym map to the 
 locus of  coverings of hyperelliptic curves. By means of the bigonal construction (see \cite{do_fibres} and \cite{pa})  we prove Theorem (\ref{main_hyp}) showing in particular the existence of  positive dimensional fibers for $r=2$ and $r=4$ and all genus $g>0$.

\section{Preliminaries}

\subsection{The differential of the Prym map and positive dimensional fibers}

By the theory of double coverings, the moduli space $\mathcal R_{g,r}$ can be alternatively described as 
the following moduli space of triples $(C, \eta, B)$:
\[
\cR_{g,r}=\{ (C, \eta,  B)  \mid  [C] \in \mathcal M_g,   \eta \in \Pic^{\frac{r}{2}}(C),  B \text{ reduced divisor in } |\eta^{\otimes 2}|  \}/\cong.
\]
The codifferential of $\cP_{g,r}$ at a point  $[(C, \eta, B)]\in \mathcal R_{g,r}$  is given by the multiplication map (\cite{lo11})
$$
d\cP_{g,r}^* (C, \eta, B): \Sym^2H^0(C, \omega_C \otimes \eta) \lra H^0(C, \omega_C^2 \otimes \cO(B)).
$$
Assume that this map is not surjective. Let us recall  Theorem $1$ in \cite{gl}: given $L$ very ample on a curve $C$ of genus at least $2$,  if 
$$
\deg(L)\ge 2g+1-2 h^1(C,L)- \text {Cliff}(C),
$$
then $\Sym^m H^0(C,L)\lra H^0(C, L^{\otimes m})$ is surjective for all $m$. In our case $L=\omega_C \otimes \eta$ has degree $2g-2+\frac r2$. 
We get that, either $L$ is not very ample, or 
\[
 \deg(L)<2g+1-2h^1(L)-\text{Cliff}(C).
\]
Since $h^1(L)=h^0(\eta^{-1})=0$  this simply says that
\[
 \text{Cliff}(C)<3-r/2.
\]
Hence, since $r>0$,  the only possibilities are:
\begin{enumerate}
\item [a)] $L=\omega_C\otimes \eta $ is not very ample. In particular,  $\deg L = 2g-2+\frac r2 \leq 2g$, hence $r\le 4$.
\item [b)] $r=2$ and Cliff$(C)= 1$.
\item [c)] $r=2, 4$  and Cliff$(C)=0$.
\end{enumerate}

Therefore we have
 
\begin{prop} \label{differential}
If $r\ge 6$ then the differential of $\mathcal P_{g,r}$ is injective in all the elements of $\mathcal R_{g,r}$.
\end{prop}

\begin{rem}
It is natural to ask whether (as in \cite{n2}) it is possible to characterize all positive dimensional fibers  of $\mathcal P_{g,r}$ by using this differential. 
More precisely, assume that there is a covering $[\pi:D \ra C]\in \mathcal R_{g,r}$ belonging to a positive dimensional fiber. Then we have seen that one
of the cases a), b) or c) above occurs. Case a) is equivalent to, either $r=2$ and $\eta = \mathcal O_C(x+y-z)$ for some points $x, y, z \in C$, 
or $r=4$ and $h^0(C,\eta )>0$.
The cases b) and  c) correspond to the coverings of hyperelliptic curves (with $r=2, 4$), or  trigonal curves (with $r=2$) or  quintic plane curves (with $r=2$). 
We prove that there are positive dimensional fibers  in the hyperelliptic case, with $r=2, 4$ (Theorem \ref{main_hyp})  and 
in the case of coverings over trigonal curves of genus $5$ (see Proposition \ref{trigonal_5} below).  
\end{rem}

\subsection{Other examples of non-injectivity}

As mentioned in the Introduction,  we find instances of positive dimensional fibers of the ramified Prym map by studying ramified coverings 
of hyperelliptic curves and they appear only for the values  $r=2$ and $r= 4$. 
The non-injectivity of $\mathcal P_{g,4}$ for $g\ge 3$ has been proved in \cite{nr}, where Nagaraj and Ramanan extend the tetragonal construction (see \cite[Proposition 9.9]{nr}) to double coverings over tetragonal curves, whose branch divisor is in the tetragonal linear series. They show that this construction,
like in the unramified case, provides two other ramified coverings (over tetragonal curves) having the same Prym variety. They prove more:

\begin{thm}(Nagaraj-Ramanan, \cite[Corollary 9.12]{nr}) For $g\ge 3$ the Prym map $\mathcal P_{g,4}$, restricted to the locus of the coverings of tetragonal curves has degree $3$.
\end{thm}

We want to point out the existence of other positive dimensional fibers for $g=5, r=2$ which do not arise  from hyperelliptic curves.
In this case there is a nice relation with the geometry of cubic threefolds. Indeed, in \cite{no} we show how $\cP_{5,2}$ can be identified with the restriction 
of the (compactified) unramified Prym map $\overline{\cP}_6: \overline{\cR}_{6}\lra \cA_5$ to the divisor $\Delta^n$ of the closure of the set of admissible 
coverings of nodal curves. Then:
\[
 \cP_{5,2}:\cR_{5,2}\hookrightarrow \Delta ^n \subset \overline{R}_6 \stackrel{\overline{\cP}_6}{\lra} \cA^{5}.
\]
The image of $\cR_{5,2}$ in $\Delta^n$ is the open set consisting of the double coverings $D_0\lra C_0$ such that both curves are irreducible with only one node; 
the node in $D_0$ maps to the node in $C_0$, away of the nodes the covering is unramified, and $p_a(C_0)=6$. The locus $\mathcal C$ of the intermediate 
Jacobians of smooth cubic threefolds is contained in $\mathcal P_{5,2}(\mathcal R_{5,2})$. Let $V$ be a generic smooth cubic threefold with Fano variety of 
lines $F(V)$. 
Then the fiber of the unramified map $\overline {\mathcal P}_6^{-1}(V)$ is $F(V)$ and $\mathcal P_{5,2}^{-1}(V)= F(V)\cap \mathcal R_{5,2}$ is 
 an irreducible curve for a general $V$ (see \cite{no}, section 2 for the details). In particular, the fiber is positive dimensional. Since the trigonal curves of genus $5$ are normalizations 
of nodal quintic plane curves, we have proved:

\begin{prop}\label{trigonal_5}
 The fibers of the restriction of $\mathcal P_{5,2}$ to $\mathcal R^{tri}_{g,2}$ are all positive dimensional. 
\end{prop}

\subsection{Dual polarizations}\label{iso_dual}

In this subsection we borrow the main result on dual polarizations from \cite{bl_dual} adapted to our situation.
Let $(A,L)$ be a polarized abelian variety of dimension $g-1+\frac r2$ and polarization type $\delta=(1,\ldots,1,2, \ldots , 2)$, where $2$ is repeated $g$ times. Then there is  a natural polarization $\widehat L$ in the dual abelian variety $\widehat A$ characterized by the property that the polarization maps:
\[
 \lambda_L:A \lra \widehat A, \qquad \qquad \lambda_{\widehat L}: \widehat A \lra \widehat{\widehat{A\, }} \cong A
\]
satisfy $\lambda_{\widehat L} \circ \lambda_L=2_{A}$ and $\lambda_{L} \circ \lambda_{\widehat L}=2_{\widehat A}$. The type of $\widehat L$ is $\delta'=(1,\ldots,1,2, \ldots , 2)$ where now  $1$ is repeated $g$ times (and therefore $2$ is repeated $\frac r2-1$ times). This construction can be done in families giving the following result:

\begin{thm}(\cite[Theorem 1.1]{bl_dual}) 
 There is a canonical isomorphism of coarse moduli spaces
 \[
  \mathcal A_{g-1+\frac r2}^{\delta} \stackrel{\cong }{\longrightarrow} \mathcal A_{g-1+\frac r2}^{\delta'}
 \]
 sending a polarized abelian variety of type $\delta$ to its polarized dual abelian variety.
\end{thm}

\vskip 5mm
\section{Global Torelli: Injectivity of $\mathcal P_{g,r}$ for $r\ge 6$ and $g\ge 1$}

The aim of this section is to prove that the ramified Prym map $\mathcal P_{g,r}$ for $r\ge 6$ and $g\ge 1$ is injective. This together with Proposition
\ref{differential}  prove our main Theorem \ref{main_embedding}.  
Given any Prym variety $(P, \Xi)$ in the image of $P_{g,r}$ we will recover the corresponding map  $\pi:D \lra C $ from the base
locus of $|\Xi|$.  The proof uses a particular description of the base locus of $\Xi$ given in  \cite{no}.  
For sake of completeness, we recall briefly some results from \cite{no} and we refer to this article for the details. 

Let $\pi: D\lra C$  be the double covering ramified in $r\geq 6 $ distinct points attached to the data $(C,\eta, B)$.
According to   \cite{mu} (Proposition p. 334) the base locus of $|\Xi|$ can be identified canonically in
 \[
 P^{can}=  \Nm_{\pi}^{-1} (\omega_C \otimes \eta)  \subset \Pic^{2g-2+\frac r2}(D) = \Pic^{g(D)-1}(D)
 \]
with the set
 \[
\widetilde{Bs}= \{ L \in P^{can} \ \mid \ \pi^*(JC) \subset \Theta^{can}_{D,-L} \}, 
\]
where $\Theta^{can}_D$ is the canonical presentation  $W^0_{g(D)-1}(D)$ of the theta divisor of $JD$ and
$\Theta^{can}_{D, -L}$ is the translation by the element $L \in \Pic^{g(D)-1}(D)$ (see \cite[Proposition 1.3]{no} for the details). 
Next we define 
\[
B_0:= \{ L =\pi^*(A)(p_1+\dots +p_{\frac r2-2})  \mid  A \in \Pic^g(C), \ p_i \in D, \  \Nm_{\pi} L \cong \omega_C \otimes \eta   \}.
\]
In \cite[Proposition 1.6]{no} the following is proved:
\begin{prop} \label{base-locus}   
The equality $B_0 = \widetilde{Bs}$ holds.
\end{prop}

Note that this result works for any triple in $\cR_{g,r}$ with $r\geq 6$.
For $g\geq 2$, the idea of the proof is that  the elements $L \in P^{can}$ in the base locus, coincide with the line bundles on $D$ whose push-forward $\pi_*L$ are 
rank two vector bundles on $C$ which are not semistable.  The equality  is then deduced from the condition of  $\pi_*L$ being unstable and a 
short exact sequence in  \cite{mu}  (proof of the Proposition in p. 338)  involving $\pi_*L$. 
In the case $g=1$ the proposition is proven directly from this short exact sequence.

\begin{thm} 
  The map $\mathcal P_{g,r}$ is injective for all $r\ge 6$ and $g\ge 1$.
 \end{thm}
 \begin{proof}
In order to reconstruct the covering $[\pi: D\ra C] \in \mathcal P_{g,r}$ from the base locus we study the birational class of $B_0$. 
We define 
\[
 \begin{aligned}
T:=  \{ (A,M) \in \Pic^g (C) \times \Pic^{(\frac r2-2)}(D)  \mid & A^{\otimes 2} \otimes \Nm_{\pi}M \cong \omega_C \otimes \eta, \\ &\text{ and } h^0(D,M)>0  \}.  
 \end{aligned}
\]
and  the projection $T\lra W^0_{\frac r2 -2}(D)$ is surjective. The fiber at $M\in W^0_{\frac r2 -2}(D)$ is the finite set (of cardinality $2^{2g}$) consisting of 
$A\in \Pic^g(C)$ such that
\[
 A^{\otimes 2}\cong   \omega_C \otimes \eta \otimes \Nm_{\pi}(M^{-1})\in \Pic^{2g-2+\frac r2-2+2-\frac r2}(C)=\Pic^{2g}(C).
\]
We claim that the generic element $M\in W^0_{\frac r2 -2}(D)$ satisfies $h^0(D,M)=1$ and is ``$\pi $-simple'', that is the effective divisor representing $M$ does not contain fibers of $\pi $. Indeed, since $g(D)=2g-1+\frac r2$ we have that $\frac r2 -2 =g(D)-2g-1 < g(D)-2$; therefore the generic element $M$ does not belong to
$W^1_{\frac r2 -2}(D)$. On the other hand  by dimensional reasons the map
\[
 C\times W^0_{\frac r2 -4}(D) \lra W^0_{\frac r2 -2}(D); \quad (x,M_0) \mapsto \pi^*(\mathcal O_C(x))\otimes M_0
\]
cannot be surjective which implies that the generic element is $\pi$-simple. Now we look to the  natural morphism $\tau: T\lra B_0$ defined by
 \[
 \tau (A,M) = \pi^*(A)\otimes M.
 \]
 We claim that this map is birational. Indeed, it is obviously surjective. Assume that $(A, M), (A'M') \in T$ have the same image under $\tau $  and that $M$ is $\pi$-simple. Then 
 \[
  \pi^*(A\otimes A'^{-1})\otimes M \cong M'.
 \]
 For any $\alpha \in \Pic(C) $ and for any $\pi$-simple $M$ we have  (see  the proof of the Proposition in p. 338 \cite{mu}):
\begin{equation} \label{exact_seq_Mumford}
 0\,\lra \alpha  \lra \pi_*(\pi^*(\alpha ) \otimes M )\lra \alpha\otimes \Nm_{\pi }(M) \otimes \eta^{-1} \lra \, 0.
\end{equation}
In our case this translates to
\[
0\,\lra A\otimes A'^{-1} \lra \pi_* M' \lra A\otimes A'^{-1}\otimes \Nm_{\pi }(M) \otimes \eta^{-1} \lra \, 0.
 \]
 Notice that the degree of the last term is $\frac r2 -2-\deg(\eta)=-2$. Thus 
 \[
 h^0(C, A\otimes A'^{-1})=h^0(C,\pi_*M)=h^0(D,M)>0.
 \]
 Since $\deg(A\otimes A'^{-1})=0$ this is only possible if $A\cong A'$ and then $M\cong M'$. This proves the claim. 

Recall that $\pi^* (JC[2]) \simeq JC[2]$ equals the kernel of the polarization map $\lambda _ {\Xi }:P \lra P^{\vee}$ and it acts 
 on $B_0$, respectively on $T$,  by
 $$
\alpha\cdot (\pi^*A (p_1 + \cdots p_{\frac{r}{2}-2} )) = \pi^*(A \otimes \alpha)  (p_1 + \cdots p_{\frac{r}{2}-2} ), \quad \alpha \cdot (A,M) = (A\otimes\alpha,  M),
 $$
for $ \alpha \in JC[2].$
The  map $\tau $ is equivariant with respect to this action and $T/ JC[2] = W^0_{\frac r2 -2}(D)$. In other words, we have
 recovered the locus 
 $W^0_{\frac r2 -2}(D)$
 up to birational equivalence. Moreover, the natural involution $L\mapsto \omega_D\otimes L^{-1}$ on 
 $\widetilde {Bs}$ and $B_0$ corresponds to an involution $\iota$ on $T$. More precisely, if $\sigma$ denotes the involution on $D$ associated to the double
 covering $\pi$,  $\iota$ is defined by
 $$
 \iota(A, M ) = (A, \sigma^*M).
 $$ 
 \begin{lem}
 The equality $\iota\circ \tau = \tau \circ \iota$ holds.
 \end{lem}
 \begin{proof}
 Let $(A,M) \in T$. Since $\omega_D= \pi^*\omega_C \otimes \cO_D(R)$, with $R$ the ramification divisor of $\pi$, we have
\begin{equation} \label{involution}
 \iota(\pi^*A\otimes M)= \pi^*(\omega_C\otimes A^{-1} ) \otimes \cO_D(R) \otimes M^{-1} 
\end{equation}
 On the other hand, $\tau (A, \sigma^*M) = \pi^*A \otimes \sigma^*M$, which equals \eqref{involution} since
 $$
 \pi^*(\omega_C\otimes A^{-2})=  \pi^*(\Nm M \otimes \eta^{-1} ) = M \otimes \sigma^*M \otimes \cO_D(-R).
 $$
 \end{proof}
 Observe that $\iota$ commutes with the action of $JC[2]$,
 therefore $\iota$ descends to an involution on the quotient $W^0_{\frac r2 -2}(D) $. 
 
 By a generalized Torelli Theorem due to Martens (see \cite{ran}, \cite{mar}),  the Brill-Noether locus   $W^0_{d}(D) $ determines the curve $D$ for $d\leq g(D)-1$.
 Furthermore, if $\Phi: W^0_{d}(D_1)  \lra  W^0_{d}(D_2) $ is a birational 
 correspondence between two curves $D_1$ and $D_2$ of genus $g(D)$, with $d\leq g(D)-2$, then $\Phi $ is induced by a birational map 
 $\varphi: D_1 \lra D_2$. In our situation,  this result implies that the involution on $W^0_{\frac r2 -2}(D) $ is completely determined by the involution on the 
 double covering $\pi: D \lra C$.  This finishes the proof.
 \end{proof}

\vskip 3mm
\section{The Prym map restricted to the coverings of hyperelliptic curves}

We study in this section the Prym map restricted to the locus of coverings of hyperelliptic curves. We denote by $\mathcal {RH}_{g,r}\subset \mathcal R_{g,r}$ 
the sublocus of the classes of coverings $D \lra C$, with $C$ a hyperelliptic curve of genus $g\ge 2$, and by $\mathcal P^h_{g,r}$ the restriction of 
$\mathcal P^h_{g,r}$ to  $\mathcal {RH}_{g,r}$.

The bigonal construction (see  \cite[Section $2$]{do_fibres}) transforms general elements of $\mathcal {RH}_{g,r}$ into elements of 
$\mathcal {RH}_{\frac r2-1,2g+2}$. For  convenience of the reader we remind some details of this construction. Given a ``tower'' of curves 
\[
 D \stackrel{\pi }{\lra} C \stackrel{f }{\lra} \mathbb P^1,
\]
where $\pi$ and $ f$ have degree $2$ we define a new ``tower''
\[
 D' \stackrel{\pi '}{\lra} C' \stackrel{f '}{\lra} \mathbb P^1,
\]
in the following way: consider $\mathbb P^1$ embedded in $C^{(2)}$ by sending a point $p$ to its fiber $f^{-1}(p)$. Then we have a fiber product diagram
\[
 \xymatrix@C=2pc@R=2pc{
 D' \, \ar@{^{(}->}[r] \ar[d]_{4:1} & D^{(2)} \ar[d]^{\pi^{(2)}} \\
\mathbb P^1 \ar@{^{(}->}[r] & C^{(2)}. 
}
\]
The natural involution on $D^{(2)}$ restricts to an involution on $D'$ and $C'$ is the quotient curve by this involution . Let $\pi': D' \lra C'$ denote the quotient map. 
If the branch locus  of $\pi$ is disjoint from the ramification locus of $f$ then $D'$ and $C'$ are smooth and irreducible. The following result is a consequence of \cite[Lemma 2.7]{do}: 

\begin{lem} Under the hypothesis above, we have 
\begin{enumerate}
 \item [a)] The bigonal construction applied to the tower  $ D' \stackrel{\pi '}{\lra} C' \stackrel{f '}{\lra} \mathbb P^1$ gives the initial tower
 $ D \stackrel{\pi }{\lra} C \stackrel{f }{\lra} \mathbb P^1$.
 \item [b)] The equalities $g(C)+g(C')=g(D)=g(D')$ hold.
 \end{enumerate}
 \end{lem}

Pantazis proves in \cite{pa} that the Prym variety of $\pi '$ is isomorphic, as polarized abelian varieties, to the dual of the Prym variety of $\pi$. 

We apply the bigonal construction to $\pi: D\lra C$, being $C$ hyperelliptic of genus $g\ge 2$ and $f:C\lra \PP^1$ the unique $g^1_2$ on $C$ 
(up to automorphisms on $\mathbb P^1$). We obtain a covering $D' \lra C'$ where $g(C')=\frac r2 -1$.  
Notice that the construction provides a $g^1_2$ linear series on $C'$. This is irrelevant 
if $\frac r2-1\ge 2$ (that is, if $r\ge 6$), but has  interesting consequences if $r=2$ or $r= 4$.
 Let $\widetilde {\mathcal R}_{1, 2g+2}$ be the isomorphism classes of  $(\pi:D\lra E, A)$, where $E$ is an elliptic curve, $\pi$ is a double covering ramified
  in $2g+2$  points and $A\in \Pic^2(E)$. Then, for $g\ge 2$, we have the following diagram:

\begin{equation} \label{diag_Pantazis}
 \xymatrix@C=2pc@R=1.8pc{
\mathcal  {RH}_{g,4}  \ar[r]^{\hskip -2mm\beta } \ar[dd]_{\mathcal P_{g,4}^h} & \widetilde {\mathcal R}_{1, 2g+2} \ar[d]^{\varphi} \\
& \mathcal {R}_{1,2g+2}  \ar[d]^{\mathcal P_{1,2g+2}} \\
\mathcal A_{g+1}^{(1,2,\ldots,2)} \ar[r]^{\cong }& \mathcal A_{g+1}^{(1,\ldots,1,2)}, 
}
\end{equation}
where $\beta$ is the birational map provided by the bigonal construction, $\varphi $ is the forgetful map and the bottom horizontal isomorphism  maps
a polarized abelian variety to its polarized dual (see  \ref{iso_dual}). The commutativity  of the diagram \eqref{diag_Pantazis} is the content of the main theorem of 
Pantazis in \cite{pa}. Since $\mathcal P_{1,2g+2}$ is injective (use \cite{ik} or Theorem (\ref{main_embedding}) above), the fiber of $\mathcal P_{g,4}^h$ can be 
identified (birationally) with the fiber of $\varphi $, which is an elliptic curve.

In the case $r=2$ the bigonal construction gives a birational map from $\mathcal {RH}_{g,2}$  to an open set of the moduli 
space $\widetilde{ \mathcal R}_ {0,2g+2}$ parametrizing pairs $(\pi:D \lra \PP^1, f:\PP^1 \stackrel{2:1}{\lra} \PP^1 )$ of double coverings with $\pi$ 
ramified in $2g+2$ points and  $f$ ramified in two points. Hence, the forgetful map  $\varphi: \widetilde{ \mathcal R}_ {0,2g+2} \lra \mathcal R_ {0,2g+2}$ 
has fibers isomorphic to $\PP^2$.  The  corresponding commutative diagram is the following:

\begin{equation*} 
 \xymatrix@C=2pc@R=1.8pc{
\mathcal  {RH}_{g,2}  \ar[r]^{\hskip -2mm\beta } \ar[dd]_{\mathcal P_{g,2}^h} & \widetilde {\mathcal R}_{0, 2g+2} \ar[d]^{\varphi} \\
& \mathcal {R}_{0,2g+2}  \ar[d]^{\mathcal P_{0,2g+2}} \\
\mathcal A_{g} \ar[r]^{\cong }& \mathcal A_{g}. 
}
\end{equation*}
where the bottom map is the isomorphism sending a principally polarized abelian variety to its dual polarized variety.
This finishes the proof of Theorem (\ref{main_hyp}). 

Notice that for $r\ge 6$ (and still $g\ge 2$) the diagram above becomes simply (there is no longer forgetful map):
\[
 \xymatrix@C=1cm@R=1.2cm{
\mathcal  {RH}_{g,r}  \ar[r]^{\hskip -5mm \beta } \ar[d]_{\mathcal P_{g,r}^h} & \mathcal {RH}_{\frac r2 -1, 2g+2} \ar[d]^{\mathcal P_{\frac r2 - 1,2g+2}^h} \\
\mathcal A_{g+\frac r2 -1}^{\delta } \ar[r]^{\cong }& \mathcal A_{g+\frac r2 -1}^{\delta '}, 
}
\]
for some convenient polarization types $\delta $ and $\delta '$ (see subsection (\ref{iso_dual})).

\end{document}